\documentclass[12pt]{article}
\usepackage{geometry}                		% See geometry.pdf to learn the layout options. There are lots.
\usepackage[parfill]{parskip}    		% Activate to begin paragraphs with an empty line rather than an indent
\usepackage{amssymb}
\usepackage{amsmath}

\usepackage{palatino}

\usepackage{graphicx}
\usepackage[english]{babel}
\usepackage{amscd}
\usepackage{amsthm}

\newtheorem{theorem}{Theorem}[section]
\newtheorem{lemma}[theorem]{Lemma}

\newtheorem{definition}{Definition}
\newtheorem{conjecture}[theorem]{Conjecture}

\DeclareMathOperator*{\diam}{diam}
\DeclareMathOperator*{\supp}{supp}
\DeclareMathOperator*{\one}{one}

\title{A fractal perspective on optimal antichains and intersecting subsets of the unit $n$-cube}

\author{Konrad Engel \thanks{Universit\"at Rostock,  Institut f\"ur Mathematik, 18051 Rostock, Germany. E-mail: konrad.engel@uni-rostock.de} 
\and
Themis Mitsis\thanks{Department of Mathematics and Applied Mathematics, University of Crete, 70013 Heraklion, Greece. 
E-mail: themis.mitsis@gmail.com}
\and
Christos Pelekis\thanks{The Czech Academy of Sciences, Institute of Computer Science, Pod Vod\'{a}renskou v\v{e}\v{z}\'{\i} 2, 182 07 Prague, Czech Republic. 
		Research supported by the Czech Science Foundation, grant number GJ16-07822Y, with institutional support RVO:67985807. E-mail: pelekis.chr@gmail.com} }

\begin{document}
	\maketitle
	
	\begin{abstract}
An \emph{$n$-cube antichain} is a subset of the unit $n$-cube $[0,1]^n$ 
that does not contain two elements 
$\mathbf{x}=(x_1, x_2,\ldots, x_n)$ and $\mathbf{y}=(y_1, y_2,\ldots, y_n)$ satisfying 
$x_i\le y_i$ for all  $i\in \{1,\ldots,n\}$. 
Using a chain partition of an adequate finite poset we show that the Hausdorff dimension of an $n$-cube antichain is at most $n-1$.
We conjecture that the $(n-1)$-dimensional Hausdorff measure of an $n$-cube antichain is at most $n$ times the Hausdorff measure of a facet of the unit $n$-cube and we
verify this conjecture for $n=2$ as well as under the assumption that the $n$-cube antichain is a smooth surface.  
Our proofs employ estimates on the Hausdorff measure of an $n$-cube antichain in terms of the sum of the  Hausdorff measures of its injective projections.  Moreover, 
by proceeding along devil's staircase, we construct a $2$-cube antichain whose $1$-dimensional 
Hausdorff measure equals $2$.  Additionally, we discuss a problem with an intersection condition in a similar setting. 
	\end{abstract}

\noindent {\emph{Keywords}: antichain; intersection condition; Sperner's theorem; Erd\H{o}s-Ko-Rado theorem; Hausdorff dimension; devil's staircase

\section{Prologue, related work and main results}\label{s1}

Let $[n]$ be the set of integers $\{1,\ldots,n\}$ and $[n]_0=[n] \cup \{0\}$. The cardinality of a finite set $F$ is denoted by $|F|$, as usual.
A collection $\mathcal{F}$ of subsets of $[n]$ having the property that no element in $\mathcal{F}$ is contained in another is referred to as an \emph{antichain}  (or \emph{Sperner family}) of $[n]$.  
A collection of subsets of $[n]$ whose cardinality equals $k$ is called a $k$-\emph{uniform family}. 
For a positive integer $t$, a collection of subsets $\mathcal{F}$ of $[n]$ is called $t$-\emph{intersecting} if 
$|A\cap B|\ge t$ for all $A,B\in \mathcal{F}$. 
 
 Let us begin with two well-known results from extremal set theory for antichains and $k$-uniform $t$-intersecting families.  
The problem of determining the maximum cardinality of an antichain is considered as the starting point of extremal set theory, which has been a fast growing area of combinatorics for several decades. 
The following result of Sperner provides a sharp upper bound on the number of elements in an antichain. \\

\begin{theorem}[Sperner \cite{sperner}] Fix a positive integer $n$ and let $\mathcal{F}$ be an antichain of $[n]$. 
Then $|\mathcal{F}|\leq\binom{n}{\lfloor n/2 \rfloor}$.
\label{th:0}
\end{theorem}

In other words, the maximum "size'' of an antichain is at most $\binom{n}{\lfloor n/2 \rfloor}$. 
Notice that the bound is sharp and that it is attained by the antichain consisting of all subsets of $[n]$ whose 
cardinality equals $\lfloor n/2 \rfloor$. 
Sperner's theorem is a fundamental result in extremal set theory that has been generalised 
in many ways (see \cite{Anderson, Engeltwo} for textbooks devoted to the topic). 
Another fundamental result from extremal set theory determines the maximum 
"size'' of a $k$-uniform $t$-intersecting family.  \\

\begin{theorem}[Erd\H{o}s-Ko-Rado \cite{erdkorado}]
Let $1\leq t\leq k$. Then there exists an integer $n_0(k,t)$ such that for all $n > n_0(k,t)$
the cardinality of a $k$-uniform $t$-intersecting family of $[n]$ is at most $\binom{n-t}{k-t}$.  
\label{ekr}
\end{theorem}

Notice that the bound is sharp and that it is attained by the family consisting of all subsets of 
$[n]$ whose cardinality equals $k$ that contain, say, the set $[t]$. This is yet another result 
in extremal set theory that has been generalised in several ways (see \cite{Anderson, Engeltwo}). 
In a celebrated paper, Ahlswede and Khachatrian \cite{AK} determined the maximum size of $k$-uniform $t$-intersecting 
families for all parameters $n,k,t$. 
In this article, we discuss "continuous versions'' of the aforementioned results. 

The idea that several combinatorial statements have continuous counterparts is rather old and several results have been reported in a "measurable'' setting (see \cite{Bollobas, Bollobas_Varopoulos,  engel, katonaone, katonatwo}) as well as in a "vector space'' setting (see \cite{FranklWilson, klainrota}). 
In this work, we consider 
continuous versions of results from extremal set theory which address the problem of determining the 
maximum Hausdorff dimension as well as the maximum Hausdorff measure of subsets on the unit $n$-cube under certain constraints that are similar to those imposed by the above mentioned theorems. 

Before being more precise, let us proceed with some observations. 
Notice that with every set $A\subseteq [n]$ one can uniquely associate a binary vector of length $n$, say $\mathbf{x}_A =(x_1,\ldots,x_n)$, where $x_i=1$ if $i\in A$ and $x_i=0$ if $i\notin A$. 
We refer to $\mathbf{x}_A$ as the \emph{characteristic vector} corresponding to $A\subseteq [n]$. 
Notice also that this correspondence is bijective and one may choose not to distinguish  between subsets of $[n]$ and 
their characteristic vectors. 
Now let $\mathcal{F}$ be an antichain of $[n]$ and let $\mathcal{B} =\{\mathbf{x}_F\}_{F\in \mathcal{F}}$ be the corresponding set of 
characteristic vectors. 
Notice that  Sperner's theorem is equivalent to the statement that 
the cardinality of the class $\mathcal{B}$ is at most $\binom{n}{\lfloor n/2 \rfloor}$ and that the class of binary vectors $\left\{\mathbf{x}=(x_1,\ldots,x_n)\in \{0,1\}^n: \sum_i x_i = \lfloor n/2 \rfloor\right\}$ attains the bound.
Now the fact that $\mathcal{F}$ is an antichain imposes certain conditions on the binary vectors of the class $\mathcal{B}$. In particular, $\mathcal{B}$ is characterized by the fact that it does 
not contain two elements  
$\mathbf{x}_F=(x_1,\ldots,x_n), \mathbf{x}_{T}= (y_1,\ldots, y_n)$ such that $x_i\le y_i$ for all $i\in [n]$. 
By relaxing the assumption that the coordinates of the characteristic vectors belong to $\{0,1\}$, one 
naturally arrives at the following definition. \\

\begin{definition}[$n$-cube antichains]
\label{def:1}
Let $n$ be  a positive integer. A subset $S$ of the unit $n$-cube $[0, 1]^n$ is called an \emph{$n$-cube antichain} if $S$ does not contain two elements $\mathbf{x}=(x_1, x_2,\ldots, x_n)$ and $\mathbf{y}=(y_1, y_2,\ldots, y_n)$ satisfying 
$x_i\le y_i$ for all $i\in [n]$.
\end{definition}
 
In this article, we shall be interested in the maximum "size'' of an $n$-cube antichain. 
Before presenting our main results, let us briefly mention a related result of the first author (see \cite{engel}). 
Let $c$ be a fixed non-negative real and assume that 
$S$ is a measurable subset of $[0,1]^n$ which does not contain two elements  $\mathbf{x}=(x_1, x_2,\ldots, x_n)$ and $\mathbf{y}=(y_1, y_2,\ldots, y_n)$ such that $x_i\le y_i$ for all $i\in [n]$ and 
 $\sum (y_i-x_i) \geq c$. 
Such a set is referred to as an \emph{$n$-cube-$c$-antichain}. 
Notice that $n$-cube-$0$-antichains are exactly the $n$-cube antichains given in Definition \ref{def:1}.
The following result determines $n$-cube-$c$-antichains of maximum "size''. \\

\begin{theorem}[\cite{engel}]
\label{th:01} Let $c>0$. 
Among all $n$-cube-$c$-antichains the set 
\[ S_c :=\left\{(x_1,\ldots,x_n): \frac{n-c}{2}\leq \sum_i x_i < \frac{n+c}{2}\right\} \] 
has maximum $n$-dimensional Lebesgue measure. 
\end{theorem}

Similarly, by relaxing the assumption that the coordinates of the characteristic vectors of a $k$-uniform $t$-intersecting family belong to $\{0,1\}$, one arrives at the following definition. \\

\begin{definition}[$(n,k,t)$-sets] Fix positive integers $n,k,t$ such that $n\geq k>t$. A 
subset $A$ of the unit $n$-cube $[0,1]^n$ is called an \emph{$(n,k,t)$-set} if every $(x_1,\ldots,x_n)\in A$ has \emph{exactly} $k$ coordinates that are strictly positive and for every two elements $(x_1,\ldots,x_n), (y_1,\ldots,y_n)\in A$, there exists $t$ distinct indices $i_1,\dots,i_t \in [n]$ 
such that $x_{i_j}=y_{i_j}>0$ for $j \in [t]$.
\end{definition}

In this article, we also address the problem of determining the maximum "size" of $(n,k,t)$-sets. 

Notice that  Theorem \ref{th:01} does not provide useful information when $c=0$.
In fact, it is easy to see using Lebesgue's density theorem that the $n$-dimensional Lebesgue 
measure of a measurable $n$-cube antichain equals zero. This suggests that the $n$-dimensional 
Lebesgue measure is not an appropriate notion of "size'' for an $n$-cube antichain and therefore it is natural 
to look at its Hausdorff dimension.  Let us briefly recall some definitions from the 
theory of fractals. If $S$ is a non-empty subset of $\mathbb{R}^n$, we denote by 
$\diam(S)$ its diameter. Fix a positive real number $s$ and, for $\delta>0$, let  
\[ \mathcal{H}_{\delta}^{s}(S) = \inf \left\{ \sum_{i}\textrm{diam}(U_i)^s: S \subseteq \bigcup_i U_i \; \textrm{and} \;  \textrm{diam}(U_i)\le \delta  \right\} .  
\]
The limit $\lim_{\delta\rightarrow 0}\mathcal{H}_{\delta}^{s}(S)$, denoted $\mathcal{H}^{s}(S)$, is the $s$-\emph{dimensional Hausdorff measure} of $S$. The \emph{Hausdorff dimension}, denoted $\dim_H S$, is defined as 
\[ \dim_H S = \inf \left\{ s : \mathcal{H}^s(S) = 0   \right\} .  \]
Finally, the \emph{upper box-counting dimension} of $S$ is defined as 
\[ 
\overline{\dim}_B S = \limsup_{\delta\rightarrow 0} \frac{\log N_{\delta}(S)}{-\log \delta} ,
\]
where $N_{\delta}(S)$ is the smallest number of sets of diameter $\delta$ which can cover $S$. 
We refer the reader to Falconer \cite{Falconer_1990} for further details. In the sequel, we will apply the well known fact (see \cite[p. 48]{Falconer_1990}) that $\dim_H S \leq \overline{\dim}_B S$. 
Our main result concerning the Hausdorff dimension of $n$-cube antichains reads as follows. \\

\begin{theorem}
\label{th:11} Fix a positive integer $n\ge 2$ and let $S\subseteq [0,1]^n$ be an $n$-cube antichain. Then the Hausdorff dimension of $S$ is at most $n-1$. 
\end{theorem}

Notice that the set $\left\{(x_1,\ldots,x_n)\in [0,1]^n: \sum_i x_i = n/2\right\}$ is an $n$-cube antichain whose Hausdorff dimension equals $n-1$ and therefore the bound in Theorem \ref{th:11} is sharp.     
Given this result, it is natural to ask 
for sharp upper bounds on the $(n-1)$-dimensional Hausdorff measure of an $n$-cube antichain. 
We are unable to settle this problem in general. We conjecture that 
the $(n-1)$-dimensional Hausdorff measure of an $n$-cube antichain is at most $n$ times the Hausdorff measure of the unit $(n-1)$-cube, i.e., $n \sigma_{n-1}$, where
\[
\sigma_n=\frac{2^{n} \Gamma(n/2+1)}{\pi^{n/2}}.
\]
In Section \ref{s2}, we verify the 
validity of this conjecture when $n=2$ as well as when $S$ is a smooth hypersurface. In particular, 
we obtain the following results.   \\

\begin{theorem}
\label{th:2} 
\renewcommand{\labelenumi}{\alph{enumi})}~
\begin{enumerate}
\item Let $S$ be a $2$-cube antichain in $[0,1]^2$. Then 
$\mathcal{H}^{1}(S) \leq 2$. 
\item There exists a $2$-cube antichain whose 
Hausdorff measure equals $2$. 
\end{enumerate}
\end{theorem}

Let $\pi_j: [0,1]^n\rightarrow [0,1]^n$ be the projections defined by 
\[
\pi_j(x_1,\dots,x_n)=(x_1,\dots,x_{j-1},0,x_{j+1},\dots,x_n), j=1,\dots,n.
\]
Under further assumptions on smoothness, the above mentioned conjecture is true. \\

\begin{theorem}\label{smooth}
Let $S$ be a smooth $n$-cube antichain. Then
\[\mathcal{H}^{n-1}(S)\leq\sum_{j=1}^n\mathcal{H}^{n-1}(\pi_j(S)).\]
In particular, 
we have $ \mathcal H^{n-1}(S)\leq n \sigma_{n-1}$.
\end{theorem}

The bound  $n \sigma_{n-1}$ for the $(n-1)$-dimensional measure of a smooth $n$-cube antichain is asymptotically sharp, as can be seen by the hypersurface 
\[
S_p=\left\{\mathbf{x}=(x_1,\dots,x_n):x_j\geq0\text{ for all $j$ and }\|\mathbf{x}\|_p^p:=\sum_{j=1}^nx_j^p=1\right\},\]
as $p\to+\infty$. Indeed, it is easily verified that as $p\to+\infty$ the $\ell^p$-unit ball $B_p=\{\mathbf{x}\in\mathbb R^n:\|\mathbf{x}\|_p\leq1\}$ converges with respect to the Hausdorff metric to the $\ell^\infty$-unit ball. However, it is well known (\cite[p. 219]{schneider}) that if a sequence of convex bodies $K_i$ converges to a convex body $K$ with respect to the Hausdorff metric, then $\mathcal H^{n-1}(\partial K_i)\to\mathcal H^{n-1}(\partial K)$. 

Our main result on the Hausdorff dimension of $(n,k,t)$-sets reads as follows. \\

\begin{theorem}\label{hausekr}
Fix positive integers $n,k,t$ such that $n\geq k >t$ and let  $A$ be an
$(n,k,t)$-set. Then the Hausdorff dimension of $A$ is at most $k-t$. 
\end{theorem}

Notice that the set consisting of all points $(x_1,\ldots,x_n)$ that have $k$ non-zero coordinates 
and whose first $t$ coordinates satisfy $x_j=\alpha_j$ for some $\alpha_j\in (0,1]$, $j \in [t]$, is an $(n,k,t)$-set whose Hausdorff dimension 
equals $k-t$. Therefore, the bound in Theorem \ref{hausekr} is sharp. Given Theorem \ref{hausekr}, it is natural 
to ask for upper bounds on the $(k-t)$-dimensional Hausdorff measure of an $(n,k,t)$-set. 
The following result implies that the sets described above have maximum $(k-t)$-dimensional measure and may be seen as a continuous analogue of the Erd\H{o}s-Ko-Rado theorem.  \\

\begin{theorem}\label{measekr} 
Fix positive integers $n,k,t$ such that $n\geq k > t$ and suppose that $A$ is an $(n,k,t)$-set. 
Then $\mathcal{H}^{k-t}(A)\leq \binom{n-t}{k-t} \sigma_{k-t}$. 
\end{theorem}

The remaining part of our article is organised as follows. 

In Section \ref{s2}, we collect the results on $n$-cube antichains. In particular, 
we prove Theorems \ref{th:11}--\ref{smooth}. The first theorem is proved using a chain partition of an adequate poset. The first statements of Theorems \ref{th:2} and \ref{smooth} are obtained via an upper estimate of the Hausdorff measure of $S$ in terms of the sum of the Hausdorff measures of its injective projections. 
The second statement of Theorem \ref{th:2} is obtained by showing that devil's staircases are examples of $n$-cube antichains. 

In Section \ref{s3}, we collect the results on $(n,k,t)$-sets. We prove Theorem \ref{hausekr} by employing  estimates 
on the Hausdorff measure of the difference set $A-A$ and  
Theorem \ref{measekr} by applying a result on nontrivial intersections of integer vectors. 
Finally, in Section  \ref{s4} we state some conjectures.

\section{$n$-cube antichains}\label{s2}

The proof of
Theorem \ref{th:11} relies on the following 
Sperner-type result for integer vectors. \\

\begin{lemma}
\label{finitesperner}
Fix positive integers $n,m\ge 1$. 
Let $\mathcal{F}$ be a collection of $n$-tuples from the set $[m-1]_0$ which does not 
contain two different tuples $(d_1,\ldots,d_n)$ and $(k_1,\ldots,k_n)$ such that 
$d_i < k_i$ for all $i\in [n]$. Then  $|\mathcal{F}| \le n m^{n-1}$. 
\end{lemma}
\begin{proof} Given an $n$-tuple $\mathbf{d}=(d_1,\ldots,d_n)\in [m-1]_0^n$ let 
$d_{(1)}= \min \{d_i: i\in [n]\}$, $d_{(n)}= \max \{d_i: i\in [n]\}$  
and consider the class $D_{\mathbf{d}}$ consisting of all $n$-tuples of the form 
\[  (d_1 - d_{(1)} +j,\; d_2 - d_{(1)} +j, \ldots,\; d_n -d_{(1)}+j), \; \textrm{where} \; j \in \{0,1,\ldots,n-(d_{(n)}-d_{(1)})\} .\] 
Observe that each class $D_{\mathbf{d}}$ forms a chain in the sense that 
whenever two $n$-tuples, say $(d_1,\ldots,d_n)$ and $(k_1,\ldots,k_n)$, belong to the same class $D_{\mathbf{d}}$ then we either have $d_i < k_i$ for all $i\in [n]$, or $d_i >k_i$ for all 
$i\in [n]$.
This implies that for every $n$-tuple $\mathbf{d}$ at most one element from $D_{\mathbf{d}}$ 
can belong to $\mathcal{F}$. Clearly, every $n$-tuple belongs to some chain 
and different chains are disjoint. Therefore the result will follow once we show that there are at most 
$n m^{n-1}$ chains. To this end, let $D_{\mathbf{d}}$ be the chain 
corresponding to  $\mathbf{d}=(d_1,\ldots,d_n)$. If the element $d_{(1)}$ is in the $\ell$-th coordinate of $\mathbf{d}$ then the $\ell$-th  
coordinate of $(d_1 - d_{(1)} ,d_2 - d_{(1)} , \ldots,d_n -d_{(1)})$ equals zero. 
This means that we can choose from every class $D_{\mathbf{d}}$ an $n$-tuple having a zero coordinate and the number of such $n$-tuples is at most $n m^{n-1}$. 
\end{proof}
Let 
\[
I_{j,m}=
\begin{cases}
[\frac{j}{m},\frac{j+1}{m}) &\text{ if } j \in [m-2]_0,\\
[\frac{j}{m},\frac{j+1}{m}] &\text{ if } j =m-1.
\end{cases}
\]

\begin{proof}
[Proof of Theorem \ref{th:11}]
Since $\dim_H S\le \overline{\dim}_B S$ for any $n$-cube antichain $S$,  it suffices to show that $\overline{\dim}_B S\le n-1$ for any $n$-cube antichain $S$. 
For each integer $m\ge 1$, write the unit $n$-cube as a union of  
cubes all of whose sides are equal to $\frac{1}{m}$.  More precisely, write the unit $n$-cube as a union 
of cubes of the form
\[
C(d_1,\ldots, d_n):=I_{d_1,m} \times I_{d_2,m} \times \dots \times I_{d_n,m},
\]
where  $d_i\in [m-1]_0$ for all $i \in [n]$. Notice that each cube $C(d_1,\ldots, d_n)$ 
can be uniquely identified by the vector $(d_1,\ldots, d_n)\in [m-1]_0^n$. 
Fix an $n$-cube antichain $S$, and let $N_{1/m}(S)$ be the number of cubes $C(d_1,\ldots, d_n)$ that have non-empty intersection with $S$. We claim that 
$N_{1/m}(S)\le n m^{n-1}$. 
Indeed, for every two different cubes  
$C(d_1,\ldots, d_n)$ and $C(k_1,\ldots, k_n)$ that have non-empty intersection with $S$ the 
corresponding $n$-tuples do not satisfy $d_i < k_i$ for all $i\in [n]$ and hence the claim follows from 
Lemma \ref{finitesperner}.

Consequently,
\[
\overline{\dim}_B S\le \limsup_{m\rightarrow \infty}\frac{\log N_{1/m}(S)}{-\log\frac{1}{m}}\le \limsup_{m\rightarrow \infty}\frac{\log (n m^{n-1})}{\log m}=n-1,
\] 
as required. 
\end{proof}

We proceed with the proof of  Theorem \ref{th:2} a).  In fact, we provide two 
proofs of this statement. 
The first proof exploits the fact that the projections of a $2$-cube antichain are injective. 

\begin{proof}[First proof of Theorem \ref{th:2} a)]
Here we denote the projections $\pi_1$ and $\pi_2$ by $\pi_x$ and $\pi_y$, respectively, since they are projections into the $x$- and $y$-axis, respectively.
We show that 
\begin{equation}\label{eq:1}
\mathcal{H}^1(S) \leq \mathcal{H}^1(\pi_x(S)) + \mathcal{H}^1(\pi_y(S))  
\end{equation}
and the result follows. 
Note that since $S$ is a $2$-cube antichain the projections are injective. 
For two sets $A,B \subseteq [0,1]$ let $A < B$ if $a < b$ for all $a \in A$ and $b \in B$.
Let $\{Q_i\}$ and $\{R_j\}$ be covers of $\pi_x(S)$ and $\pi_y(S)$, respectively, where without loss of generality
$Q_1 < Q_2 < \cdots$ and $R_1 < R_2 < \cdots$ Define the sets 
$Q_{ij}=\{x \in Q_i: \exists y \in R_j \text{ such that } (x,y) \in S\}$ and 
$R_{ij}=\{y \in R_j: \exists x \in Q_i \text{ such that } (x,y) \in S\}$. 
Then 
\[
S \subseteq \bigcup_{ij} (Q_{ij} \times R_{ij}).
\]
Therefore, using the triangle inequality,
\begin{equation}
\label{cover}
\mathcal{H}^1(S) \le \sum_{ij}\diam (Q_{ij} \times R_{ij}) \le \sum_{ij}\diam (Q_{ij}) + \sum_{ij} \diam(R_{ij}).
\end{equation}
Since $S$ is a $2$-cube antichain $Q_{i1} > Q_{i2} > \cdots$ and $R_{1j} > R_{2j} > \cdots$ for all $i,j$.
This implies that 
\begin{align*}
\sum_{j}\diam (Q_{ij}) &\le \diam(Q_i)\quad \text{for all }i,\\
\sum_{i}\diam (R_{ij}) &\le \diam(R_j)\quad \text{for all }j.
\end{align*}
Together with \eqref{cover} this yields
\[
\mathcal{H}^1(S) \le \sum_{i}\diam (Q_{i}) + \sum_{j} \diam(R_{j}).
\]
The result follows by taking the infimum with respect to all covers.
\end{proof}

We continue with a second proof of Theorem \ref{th:2} a), which is based upon the following well-known result regarding the Hausdorff measure of the image of Lip\-schitz functions, cf. \cite[p. 24]{Falconer_1990}. Recall that a function 
$f:F \subseteq \mathbb{R}^n\to\mathbb{R}^m$ is 
\emph{Lipschitz with constant $c$} if 
\[ |f(\mathbf{x})-f(\mathbf{y})| \leq c\cdot |\mathbf{x}-\mathbf{y}|  \; \text{for all} \; \mathbf{x},\mathbf{y}\in F . \]

\begin{lemma}\label{Lipschitz_lemma}
Fix positive integers $n,m$ and let $F\subseteq \mathbb{R}^n$. If $f: F\to\mathbb{R}^m$ is a Lipschitz function with constant $c$ then $\mathcal{H}^s(f(F))\leq c^s \mathcal{H}^s(F)$.
\end{lemma}

The following proof exploits the fact that the "diagonal projections" of a Sperner subset of the 
unit square are injective whose inverse is Lipschitz of constant $1$. 

\begin{proof}[Second proof of Theorem \ref{th:2} a)]
Denote $Q_1 = [0,1]^2\cap \{(x,y): x\ge y\}$ and $Q_2 = [0,1]^2\setminus Q_1$.  
Let $S_1= Q_1 \cap S$ and $S_2 = Q_2\cap S$. 
Now consider the diagonal projection $\phi_1$ of $S_1$ into $[0,1]\times \{0\}$ 
defined by $\phi_1(x,y) = (x-y,0)$. 
Notice that the assumption that $S_1$ is a $2$-cube antichain implies that $\phi$ is a bijection of $S_1$ onto its image.
Let $(a,0),(b,0)\in \phi_1(S_1)$ with $a<b$ and look at the inverse images,  $\phi_1^{-1}(a,0), \phi_1^{-1}(b,0)$.
Using again that $S_1$ is a $2$-cube antichain, it follows that these are of the form  
$\phi_1^{-1}(a,0) = (x,y)$ and $\phi_1^{-1}(b,0) = (x+\epsilon_1, y-\epsilon_2)$, for some  $x,y\in [0,1]$ and $\epsilon_1,\epsilon_2>0$. 
Now notice that 
\[ |\phi_1^{-1}(a,0)- \phi_1^{-1}(b,0) | = \sqrt{\epsilon_1^2 + \epsilon_2^2} \leq \epsilon_1+\epsilon_2 = |a-b| .  \]
This implies that $\phi_1^{-1}$ is Lipschitz with constant $1$ and  Lemma \ref{Lipschitz_lemma} yields
\[\mathcal{H}^1(S_1)=\mathcal{H}^1(\phi_1^{-1}(\phi_1(S_1))) \leq \mathcal{H}^1(\phi_1(S_1)).\]  
The analogous argument with $\phi_2:S_2\to[0,1] \times \{0\}$, $\phi_2(x,y)=(y-x,0)$, yields 
\[\mathcal H^1(S_2)\leq\mathcal H^1(\phi_2(S_2)).\]
Therefore $\mathcal H^1(S)\leq\mathcal H^1(\phi_1(S_1))+\mathcal H^1(\phi_2(S_2))\leq 2$.
\end{proof}

It remains to show part b) of Theorem \ref{th:2}, i.e., that 
the bound is sharp. The proof of this statement requires the following result from measure theory (see \cite[Proposition 5.5.4]{Bogachev}). \\

\begin{lemma}\label{bogachev} Let $f:[0,1]\to [0,1]$ be a function and let $E$ be a measurable set such that at every point of $E$ the function $f(\cdot)$ is differentiable. 
Then 
\[ \lambda(f(E)) \leq \int_E |f'(x)|\; dx , \]
where $\lambda(\cdot)$ denotes the $1$-dimensional Lebesgue measure. 
\end{lemma} 

Now we can prove the existence of "maximum'' $2$-cube antichains in the unit square. 

 \begin{proof}[Proof of Theorem \ref{th:2} b)]
Let $f:[0,1]\to[0,1]$ be a continuous, strictly increasing function having zero derivative almost everywhere. 
An example of such a function can be found in \cite{Zaanen_Lux} and is referred to as \emph{devil's staircase}.  
We divide the graph $S:=\{(x,f(x)):x\in[0,1]\}$ into two parts, namely, 
$A=\{(x,f(x): f'(x)=0)\}$ and $B = S\setminus A$. 
Since $f'=0$ almost everywhere, the projection of $A$ into the $x$-axis has measure $1$ and so 
$\mathcal{H}^1(A) \ge 1$. 
From Lemma \ref{bogachev}, it follows that the set $f(\{x:f'(x)=0\})$ has measure zero, which in turn implies that the projection of $B$ into the $y$-axis has measure $1$. Thus $\mathcal{H}^1(B) \ge 1$.
Putting these two bounds together, we conclude 
\[ \mathcal{H}^1(S) = \mathcal{H}^1(A) + \mathcal{H}^1(B) \ge 2  \]
and therefore, by (\ref{eq:1}), we have $\mathcal{H}^1(S)=2$. 
Now the result follows by observing that the function $1-f(\cdot)$ is strictly decreasing 
and therefore its graph is a $2$-cube antichain. 
\end{proof}

In higher dimensions, the situation is not as satisfactory. Using the argument in Theorem \ref{th:11}, one readily shows that if $S$ is an $n$-cube antichain, then $\mathcal H^{n-1}(S)\leq n\cdot n^{\frac{n-1}2}$, which is far away from the conjectured $\mathcal H^{n-1}(S)\leq n \cdot \sigma_{n-1}$. However, the conjecture can be verified in the case of smooth hypersurfaces.  

\begin{proof}[Proof of Theorem \ref{smooth}] 
If $\mathbf{x}=(x_1,\ldots,x_{n-1})$ is a vector in the cube $[0,1]^{n-1}$, 
let $\mathbf{x}_j$ denote the vector $(x_1,\ldots,x_{j-1},x_{j+1},\ldots,x_{n-1})$, for $j=1,\ldots,n-1$.
We note that the projections $\pi_j$ restricted to $S$ are injective, and there exists a smooth $f:\pi_n(S)\to[0,1]$ whose graph is $S$. 
For $j=1,\dots,n-1$, the function $T_j(\mathbf x)=(\mathbf x_j,f(\mathbf x_j))$ is injective, the absolute value of its Jacobian is $|\frac{\partial f}{\partial x_j}|$, and $T_j(\pi_n(S))=\pi_j(S)$. Hence, using the area formula (see \cite[Section $3.3.4$]{Evans_Gariepy}), the surface area of $f$ can be estimated as 

\begin{eqnarray*}
\frac{1}{\sigma_{n-1}}\mathcal{H}^{n-1}(S) &=&\int_{\pi_n(S)}(1+|\nabla f|^2)^{1/2}\\
 &\leq& \mathcal L^{n-1}(\pi_n(S))+\int_{\pi_n(S)}|\nabla f|\leq \mathcal L^{n-1}(\pi_n(S))+\sum_{j=1}^{n-1}\int_{\pi_n(S)}\left|\frac{\partial f}{\partial x_j}\right| \\
&=&\mathcal L^{n-1}(\pi_n(S))+\sum_{j=1}^{n-1}\int_{T_j(\pi_n(S))}1=\sum_{j=1}^n\mathcal L^{n-1}(\pi_j(S)) .
\end{eqnarray*}
Hence
\[
\mathcal{H}^{n-1}(S) \le \sum_{j=1}^n\mathcal{H}^{n-1}(\pi_j(S)),
\]
and the result follows. 
\end{proof}

\section{$(n,k,t)$-sets}\label{s3}

In this section, we collect our results regarding $(n,k,t)$-sets. We begin with the proof of
Theorem \ref{hausekr} using difference sets. We note that the result can be also easily derived from the proof of Theorem \ref{measekr}.

\begin{proof}[Proof of Theorem \ref{hausekr}]
Given a subset $F$ of $[n]$ with cardinality $k$, let $[0,1]^n_F$ denote the cartesian 
product $X_1\times \cdots \times X_n$, where $X_i = \{0\}$ when $i\notin F$, and 
$X_i = [0,1]$ when $i\in F$. 
Now let $A$ be any $(n,k,t)$-set. Clearly, $A$ is contained in $\cup_F [0,1]^n_F$, where the union runs over all subsets of $[n]$ of cardinality $k$. Since the Hausdorff dimension is stable under finite unions (see \cite[Chapter 3]{Falconer_1990}), it is enough to show that the Hausdorff dimension of $A_F = A \cap [0,1]^n_F$ is at most $k-t$. 
We may consider $A_F$ as a $(k,k,t)$-set (in the unit $k$-cube $[0,1]^k$) and from now on we write $A$ instead of $A_F$. 
Recall that the difference set $A-A$ is defined as $A-A=\{\mathbf{a}-\mathbf{b}: \mathbf{a},\mathbf{b}\in A\}$. 
Notice that for every 
$\mathbf{x}=(x_1,\ldots,x_k), \mathbf{y}=(y_1,\ldots,y_k)\in A$, there exist distinct indices $i_1,\dots,i_t\in [k]$ so that 
the $x_{i_j}-y_{i_j}=0$, $j \in [t]$, i.e., $\mathbf{x}-\mathbf{y} \in [0,1]^k_{[k] \setminus \{i_1,\dots,i_t\}}$.  
This implies that 
\[ A-A \subseteq \bigcup_{I\subseteq [k]: |I|=t} [0,1]^k_{I} . \]
Obviously, for each $I$ with $|I|=t$, we have $\dim_H([0,1]^k_{[k] \setminus I}) = k-t$ and hence $\dim_H(A-A) \le k-t$.
Fix some  $a\in A$. Then $a-A\subseteq A-A$ and therefore 
\[
\dim_H(A) = \dim_H(a-A) \le \dim_H(A-A) \le k-t,
\]
as required. 
\end{proof}

For the proof of Theorem \ref{measekr} we need the following result of Bey and the first author \cite{BE} which was reproved in a different way by P.L. Erd\H{o}s, Seress and Sz\'{e}kely \cite{ESS}.
In order to formulate the result we need some further notations and definitions.
For $\mathbf{d} \in [m]_0^n$, let $\supp(\mathbf{d})=\{i \in [n]: d_i > 0\}$ be the \emph{support} of $\mathbf{d}$
and let $\one(\mathbf{d})=\{i \in [n]: d_i = 1\}$ be the \emph{index set of ones} of $\mathbf{d}$.
Let 
\[
H_{n,k,m} = \{\mathbf{d} \in [m]_0^n: |\supp(\mathbf{d})|=k\}. 
\]

A family $\mathcal{F} \subseteq  [m]_0^n$  is called   \emph{$k$-uniform $t$-intersecting}
if $\mathcal{F} \subseteq H_{n,k,m}$ and for all $\mathbf{d},\mathbf{e} \in F$ there are distinct indices $i_1,\dots,i_t \in \supp(\mathbf{d}) \cap \supp(\mathbf{e})$ such that $d_{i_j}=e_{i_j}$ for all $j \in [t]$.
The $k$-uniform $t$-intersecting family $F$ is called \emph{trivial} if there are distinct indices $i_1,\dots,i_t$ and numbers $a_1,\dots,a_t \in [m]$ such that
$d_{i_j}=a_j$ for all $j \in [t]$, otherwise it  is called \emph{non-trivial}. Examples of non-trivial $k$-uniform $t$-intersecting families are
\begin{align*}
F_1&=\{\mathbf{d} \in H_{n,k,m} \text{ and } [t] \subseteq \one(\mathbf{d}) \text{ and }\one(\mathbf{d}) \cap \{t+1,\dots,\min(k+1,n)\} \neq \emptyset\} \\
&\cup \{\mathbf{d} \in H_{n,k,m} \text{ and } |\one(\mathbf{d})\cap [t]|=t-1 \text{ and } \{t+1,\dots,\min(k+1,n)\}  \subseteq \one(\mathbf{d})\},\\
F_2&=\{\mathbf{d} \in H_{n,k,m} \text{ and } |\one(\mathbf{d}) \cap[t+2]| \ge t+1 \}.
\end{align*}

\begin{theorem}[\cite{BE}, \cite{ESS}] \label{subcubesekr} 
Fix positive integers $n,k,t$ such that $n\geq k > t$. There is an integer $m_0(n,k,t)$ such that for all $m > m_0(n,k,t)$ every non-trivial $k$-uniform $t$-intersecting family in $[m]_0^n$ has size bounded as follows:
\[
|\mathcal{F}| \le \max(|F_1|,|F_2|).
\]
\end{theorem}
Note that for fixed $n,k,t$
\begin{equation}
\label{bound}
 \max(|F_1|,|F_2|)=O(m^{k-t-1}) \text{ as } m \rightarrow \infty
\end{equation}
since there is a bounded number of choices for $\one{\mathbf{d}}$ and $|\one{\mathbf{d}| \ge t+1}$, i.e., there are at most ${k-t-1}$ free positions if the index set of ones is fixed.
Now we are ready to prove Theorem \ref{measekr}.

\begin{proof}[Proof of Theorem \ref{measekr}] 
Let $A$ be an $(n,k,t)$-set. 
Notice that if there exists a $t$-element set $T=\{i_1,\dots,$ $i_t\}$ $\subseteq [n]$ and real numbers $\alpha_1,\dots,\alpha_t \in (0,1]$ such that, for all $\mathbf{d} \in A$ and for all $j \in [t]$, we have
$d_{i_j}=\alpha_j$ then $A$ is contained in the disjoint union of sets
$\{\mathbf{d} \in [0,1]^n_F: d_{i_j}=\alpha_j$, $j \in [t]\}$, where $F$ is such that $T \subseteq F \subseteq [n]$ and $|F|=k$. These sets can be considered as unit $(k-t)$-cubes
and thus have Hausdorff measure $\sigma_{k-t}$. Since there are  $\binom{n-t}{k-t}$ choices for $F$ the inequality $\mathcal{H}^{k-t}(A) \le \binom{n-t}{k-t}\sigma_{k-t}$ follows.

Now assume the contrary for $A$, i.e., for each $t$-element set $T \subseteq [n]$ there exists  an index $i_T \in T$ and there exist elements $\mathbf{d},\mathbf{e} \in A$
such that $d_{i_T} \neq e_{i_T}$.
Let
\[
\delta = \min\{|d_{i_T} - e_{i_T}|: T \subseteq [n], |T|=t\}
\]
and consider integers $m > \max(\frac{1}{\delta},m_0(n,k,t))$, where $m_0(n,k,t)$ is the constant given  by Theorem \ref{subcubesekr}. 
Analogously to the proof of Theorem \ref{th:11} we take cubes of the form
\[
C(d_1,\ldots, d_n):=J_1\times J_2 \times \dots \times J_n,
\]
where, for all i, 
\[
J_i=\begin{cases}
I_{d_i-1,m}&\text{ if } d_i>0,\\
\{0\}&\text{ if } d_i=0
\end{cases}
\]
and $\mathbf{d} = (d_1,\ldots,d_n) \in H_{n,k,m}$. 
Let
$\mathcal{C}_m(A)$ be the set of cubes $C(d_1,\ldots, d_n)$  that have non-empty intersection with $A$ and let $\mathcal{F}_m$ be the corresponding set of $n$-tuples $(d_1,\dots,d_n)$.
Clearly, $\mathcal{C}_m(A)$ is a cover of $A$ and the assumption $\frac{1}{m}<\delta$ implies that 
$\mathcal{F}_m$ is a non-trivial $k$-uniform $t$-intersecting family in $[m]_0^n$.
Using Theorem \ref{subcubesekr} and (\ref{bound}), we obtain as in the proof of Theorem \ref{th:11}
\[
\overline{\dim}_B A\le \limsup_{m\rightarrow \infty}\frac{\log |\mathcal{F}_m|}{-\log\frac{1}{m}}\le \limsup_{m\rightarrow \infty}\frac{\log (O(m^{k-t-1}))}{\log m} \le k-t-1.
\] 
Consequently,
\[
\mathcal{H}^{k-t}(A) = 0
\]
and the result follows. 
\end{proof}

\section{Conjectures}\label{s4}

In this final section, we emphasize two conjectures which, 
from our point of view, are really challenging and interesting. \\

\begin{conjecture}
\label{projections}
Let $S$ be a subset of the unit $n$-cube. If all projections $\pi_j, j\in [n]$, are injective then 
\[ \mathcal{H}^{n-1}(S) \leq \sum_{j=1}^{n} \mathcal{H}^{n-1}(\pi_j(S)). \] 
In particular, the $(n-1)$-dimensional Hausdorff measure of an $n$-cube antichain is at most $n \sigma_{n-1}$. 
\end{conjecture}

Theorem \ref{smooth} implies that the bound $n \sigma_{n-1}$ in Conjecture \ref{projections} is asymptotically sharp but we believe that the following stronger statement is true. \\

\begin{conjecture} For all $n \ge 3$ there exists an $n$-cube antichain whose $(n-1)$-dimensional 
Hausdorff measure equals $n\sigma_{n-1}$.
\end{conjecture}

\end{document}